\documentclass{article}
\usepackage{amsmath, amssymb, amsthm,xcolor}
\usepackage[authoryear]{natbib}
\newtheorem{theorem}{Theorem}[section]
\newtheorem{definition}[theorem]{Definition}
\newtheorem{lemma}[theorem]{Lemma}
\newtheorem{proposition}[theorem]{Proposition}
\newtheorem{example}[theorem]{Example}
\newtheorem{remark}[theorem]{Remark}

\title{ The $\alpha$-dependence of stochastic differential equations driven by variants of $\alpha$-stable processes}

\author{Jos{\'e} Lu{\'\i}s da Silva$^{1}$\thanks{Corresponding author }  \; and Mohamed Erraoui$^{2}$\\
$^{1}$ Centre of Exact Sciences and Engineering, \\ CCM, University of Madeira, 9000-390 Funchal, Portugal.\\
Email: luis@uma.pt\\
$^{2}$ Universit{\'e} Cadi Ayyad, Facult\'{e} des Sciences Semlalia, \\
D\'{e}partement de Math\'{e}matiques, BP 2390, Marrakech, Maroc\\
Email: erraoui@ucam.ac.ma
}

\begin{document}
\date{}
\maketitle

\begin{abstract}
In this paper we investigate two variants of $\alpha$-stable processes, namely tempered stable subordinators and 
modified tempered stable process as well as their renormalization.  We study the weak convergence in the
 Skorohod space and prove that they satisfy the uniform tightness condition. Finally, applications to the $\alpha$-dependence 
of the solutions of SDEs driven by these processes are discussed.

\noindent {\bf Keywords}: L{\'e}vy processes, Uniform tightness, Skorohod space, Weak convergence, SDEs. 
\end{abstract}

\section{Introduction}

In the last years, L{\'e}vy processes have received a great deal 
of attention fuelled by numerous  applications. First of all, we would like to mention the stochastic finance theory, 
 one of the principal subjects is the capital asset pricing model where the security price is allowed to have jumps, both big and small. 
Another reason to use models with jumps still in finance is for example in the stock market the price does not change continuously
 but change by units; the market is closed on weekends, holidays and opening prices often have jumps. 
We refer the interested reader to \cite{Rama-Tankov} for an introduction to some financial models driven by L{\'e}vy processes. 
Second, in electrical engineering it is known that the telephone noise is non-Gaussian and the noise is modeled by a L{\'e}vy process. 
Indeed, the work of \cite{SK} proposes to model the telephone noise  by a stable process as well as a L{\'e}vy process 
with both jumps and a Wiener component. The latter model is suggested by the different sources of noise, specifically thermal noise 
corresponds to the Wiener part and the jump term comes from possibly thunderstorm. As a third example where 
a stochastic differential equation driven by a L{\'e}vy process appears we mention the model of an infinite 
capacity dam subject to an additive input process and a general release rule. The dynamics of the content of the dam is given by

\begin{equation}
 dX_{t}=r(X_{t})dt+dZ_{t},\label{Fev11-eq1}
\end{equation}
where $Z$ is a L{\'e}vy process with nonnegative increments, $r(x)$ the release rate when the dam content is $x$. 
It has been suggested using empirical data  that the L{\'e}vy measure of $Z$ is the gamma measure, cf.~Moran (1969),
 see also \cite{PT} for the numerical schemes of such models. 
 When $Z$ stands for a NIG process, then equation (\ref{Fev11-eq1}) 
was proposed as a generalized Hull-White model in finance, see \cite{HM}.

In this paper we study the $\alpha$-dependence of the solutions of the stochastic differential equations (SDEs)
 driven by variants of $\alpha$-stable processes.

Firstly, we consider the classes of tempered stable subordinators $X^{TSS}_{\alpha}$ and 
modified tempered stable processes $X^{MTS}_{\alpha}$ with $\alpha \in (0,1/2)$
and prove the weak convergence,  in the Skorohod space endowed with the Skorohod topology,
 of $X^{TSS}_{\alpha}$ (resp.~$X^{MTS}_{\alpha}$) to the gamma process when $\alpha\rightarrow 0$ 
(resp.~the normal inverse Gaussian process $X^{NIG}$ when $\alpha\rightarrow 1/2$).
The family
$\{ X^{MTS}_{\alpha},\allowbreak \alpha \in (0,1/2)\}$ was considered in \cite{KRCB} to develop the GARCH option price model. 
We want also to point out that the weak convergence of $X^{TSS}_{\alpha}$ to the gamma process was first established 
in \cite{TVY2001} using other considerations.
Indeed it is proved that the gamma process has been obtained as weak limit of renormalized stable processes. 
The family of the renormalized stable processes we identify as the family of tempered stable processes 
$\{ X^{TSS}_{\alpha},\, \alpha \in (0,1/2)\}$, see Remark \ref{TVYprocess} below.
Moreover, we would like also to mention the work of \cite{Ryd1} where an approximation of 
the NIG process, based on an appropriate discretization of the L{\'e}vy measure, was discussed.  
In this paper, instead we use the modified tempered stable process as an approximation of the NIG process.

The continuous dependence consists in investigating the conditions under which the solutions converge weakly.
However, it is well known that the weak convergence is not sufficient to ensure the convergence of stochastic 
integral , see \cite{KP3} and references therein. Among the sufficient conditions 
we cite the uniform tightness (\textbf{UT}), introduced by \cite{S84}.
It should be noted that this condition has been used extensively to establish the results of stability of stochastic differential
equations since its introduction, see for example \cite{J}, \cite{JP}, \cite{JMP}, \cite{KP3} and \cite{MS}.
Thus we show that both driven families $\{ X^{TSS}_{\alpha},\, X^{MTS}_{\alpha},\, \alpha \in (0,1/2) \}$ 
satisfy the (\textbf{UT}) condition. This allows us to establish the continuous dependence result of SDEs driven by these families.

Secondly, it is proven in \cite{As-Ro} that the standard Brownian 
motion $\{W(t),\, t\in [0,1]\}$ is obtained as a weak limit, in the Skorohod space equipped with the uniform metric, 
 of a suitable renormalization of certain classes of L{\'e}vy process which 
includes the family $\{X^{MTS}_{\alpha},\, \alpha \in (0,1/2\}$. 
More precisely, the Brownian motion $W$ can be approximated by an appropriate renormalization of the compensated sum of
small jumps of a given L{\'e}vy process, see Proposition~\ref{21juiprop} below.
In the same spirit we mention the work \cite{covo} which completes, in some sense the previous one,
where it is shown that the process $\{t,\, t\in [0,1]\}$ is a weak limit of a renormalized (in an appropriate sense)
 sum of small jumps of classes of subordinator. 
We note that the family $\{X^{TSS}_{\alpha},\, \alpha \in (0,1/2\}$ is among those classes. 
These two results lead us to consider the dependence problem of SDEs driven by these renormalized processes.
The main tools we use to prove the continuous dependence result are the uniform tightness of the renormalized families
 and the stability of SDEs established in \cite{MS}.

\section{L{\'e}vy processes and infinite divisibility}
We start by recalling a few well-known facts about infinitely divisible distributions.
We consider a class of Borel measures on $\mathbb{R}$ satisfying
the following conditions:

\begin{eqnarray}
\Lambda(\{0\}) & = & 0,\label{Juin23eq1}\\
\int_{-\infty}^{+\infty}(s^{2}\wedge1)d\Lambda(s) & < & \infty.\label{Juin23eq2} \end{eqnarray}
This class will be denoted by $\mathfrak{M}$.

\cite{DeF} introduced the notion of an infinitely divisible
distribution and showed that they have an intimate relationship with
L{\'e}vy processes. By the L{\'e}vy-Kintchine formula, all infinitely divisible distributions $F_{\Lambda}$ are described via
their characteristic function: 
\[
\phi_{\Lambda}(u)=\int_{-\infty}^{+\infty}e^{iux}dF_{\Lambda}(x)=e^{\varPsi_{\Lambda}(u)}, ~u\in \mathbb{R},
\]
where the characteristic exponent $\varPsi_{\Lambda}$, is given as 
\[
\varPsi_{\Lambda}(u)=i b u-\frac{1}{2}cu^{2}+\int_{-\infty}^{+\infty}(e^{ius}-1-ius\,1\!\!1_{\{ |s|<1\}}(s))\, d\Lambda(s),
\]
where $b\in \mathbb{R}, c\geq0$. 

We assume as given a filtered probability space $(\Omega,\, \mathcal{F},\, P,\, (\mathcal{F}_{t})_{t\in [0,1]}) $ satisfying the 
usual hypothesis. 
\noindent A L{\'e}vy process $X=\{ X(t),t\in [0,1]\}$ has the property 
\[
\mathbb{E} (e^{iu X(t)})=e^{t\varPsi(u)},~ t\in\left[0,1\right], ~u\in \mathbb{R},
\]
 where $\varPsi(u)$ is the characteristic exponent of $X(1)$ which has an infinitely divisible distribution.  
Thus, any infinitely divisible distribution $F_{\Lambda}$
generates in a natural way a L{\'e}vy process $X$ by setting the law of $X(1)$, $\mathcal{L} (X(1))=F_{\Lambda}$. 
The three quantities $(b, c, \Lambda)$ determine the law $\mathcal{L} (X(1))$.
Since the distribution of a L{\'e}vy process $X=\{ X(t),t\in [0,1]\}$ is completely determined
by the marginal distribution $\mathcal{L} (X(1))$, and thus the process $X$ itself completely.
 The measure $\Lambda$ is called the L{\'e}vy measure whereas $(b, c, \Lambda)$ is called the L{\'e}vy-Khintchine triplet.

Let us now give some examples of L{\'e}vy processes which will be used later on. We will present three classes related to the sample 
path properties. Namely subordinators, processes with paths of finite and infinite variations. 
\subsection{Subordinators}
A subordinator is a one-dimensional increasing L{\'e}vy process starting from $0$. 
Subordinators form one of the simplest family of L{\'e}vy processes. The law of a subordinator is specified by the
Laplace transform of its one dimensional distributions. We assume throughout this paper that these processes have no drift. 

We consider a subclass in $\mathfrak{M}$ of measures supported on $\mathbb{R}_{+}$ satisfying the following 
\begin{eqnarray}
\Lambda(0,\infty) & = & \infty,\label{Mai31eq3}\\
\int_{0}^{1}s\,d\Lambda(s) & < & \infty.\label{Mai31eq4}
\end{eqnarray}
Any L{\'e}vy measure $\Lambda$ satisfying conditions (\ref{Mai31eq3}) and (\ref{Mai31eq4}) 
generates a subordinator $X$, see for example \cite[Theorem 1.2]{Ber}. 
We can therefore give its Laplace transform 
\[
\psi_{\Lambda}(u):=\mathbb{E} (e^{-u X(1)})= \exp\left(\int_{0}^{\infty}(e^{-su}-1)\, d\Lambda(s)\right),\quad u\in\mathbb{R}_{+}.
\]
\begin{remark}

\begin{description}
 \item[\textbf{(i)}]When $X$ is a subordinator, the Laplace transform of its marginal distributions is much more useful,
 for both theoretical and practical applications, than the characteristic function. 
 \item[\textbf{(ii)}]
The assumption (\ref{Mai31eq3}) implies that the process $X$ has infinite
activity, that is, almost all paths have infinitely many jumps along
any time interval of finite length. Whereas the condition (\ref{Mai31eq4})
guarantees that almost all paths of $X$ have finite variation. 
 \end{description}
\end{remark}
 \subsection*{Examples} 

\noindent \begin{enumerate}
\item \textbf{Gamma process}. Consider the L\'{e}vy measure $\Lambda_{\gamma}$ with
density with respect to the Lebesgue measure defined by\[
d\Lambda_{\gamma}(s):=\frac{e^{-s}}{s}1\!\!1_{\{s>0\}}ds.\]
Then the corresponding process is known as gamma process. A simple
calculation shows that\[
\psi_{\Lambda_{\gamma}}(u)=\frac{1}{1+u},\]
and the Laplace transform of the corresponding process
has the form\[
\mathbb{E}_{\mu_{\gamma}}\left(e^{-uX_{\gamma}(t)}\right)=\exp\left(-t\log(1+u)\right)=\frac{1}{(1+u)^{t}},~t\in [0,1].\]
Here $\mu_{\gamma}$ denotes the law  of $X_{\gamma}(1)$. 

\item \textbf{~Stable subordinator (SS)}. Let $\alpha\in(0,1)$ be given
and let $\Lambda^{SS}_{\alpha}$ be the L{\'e}vy measure given by\[
d\Lambda^{SS}_{\alpha}(s):=\frac{\alpha}{\Gamma(1-\alpha)}\frac{1}{s^{1+\alpha}}1\!\!1_{\{s>0\}}ds.\]
Then we have
\[\psi_{\Lambda}(u)=\exp\left(-u^{\alpha}\right),
 \]
and 
\[
\mathbb{E}_{\mu^{SS}_{\alpha}}\left(e^{-uX^{SS}_{\alpha}(t)}\right)=\exp\left(-tu^{\alpha}\right),~t\in [0,1].\]
Here $\mu^{SS}_{\alpha}$ denotes the law of $X^{SS}_{\alpha}(1)$.
\item \textbf{~Tempered stable subordinator (TSS).} A tempered stable subordinator is obtained by taking a 
stable subordinator and multiplying the L\'{e}vy measure by an exponential function, that is, an exponentially tempered 
version of the stable subordinator. More precisely, for $\alpha\in(0,1)$, we consider the L\'{e}vy measure 
\begin{equation}\label{equiv}
 d\Lambda_{\alpha}^{TSS}(s):=\frac{1}{\alpha} e^{-s}d\Lambda^{SS}_{\alpha}(s)=\frac{1}{\Gamma(1-\alpha)}\frac{e^{-s}}{s^{1+\alpha}}1\!\!1_{\{s>0\}}ds.
\end{equation}
Then we have\[
\psi_{\Lambda_{\alpha}^{TSS}}(u)=\exp\left(\frac{1-(1+u)^{\alpha}}{\alpha}\right)\]
and \[
\mathbb{E}_{\mu^{TSS}_{\alpha}}\left(e^{-uX_{\alpha}^{TSS}(t)}\right)=\exp\left(-t\frac{1-(1+u)^{\alpha}}{\alpha}\right),~t\in [0,1].\]
\end{enumerate}
Now let us give a concrete realization of a subordinator due to \cite{TVY2001}.
We denote by \[
D=\left\{ \eta=\sum z_{i}\delta_{x_{i}},\; x_{i}\in[0,1],\; z_{i}\in\mathbb{R_{+}},\;\sum|z_{i}|<\infty\right\} \]
the real linear space of all finite real discrete measures in $[0,1]$.
We define the coordinate process $\left\{ X(t),t\in\left[0,1\right]\right\} $  on $D$ by \[
X(t):D\longrightarrow\mathbb{R_{+}},\;\eta\mapsto X(t)(\eta):=\eta([0,t]),\quad t\in[0,1]\]
and $\mathcal{F}_{t}:=\sigma(X(s),\; s\leq t)$ denotes its own filtration. 

Let $\Lambda$ be a L{\'e}vy measure satisfying conditions (\ref{Mai31eq3}) and (\ref{Mai31eq4}) and $\mu_{\Lambda}$ be 
a probability measure on $(D,\mathcal{F}_{1})$ with Laplace transform given by \[
\mathbb{E}_{\mu_{\Lambda}}\left(\exp\left(-\int_{0}^{1}f(t)\, d\eta(t)\right)\right)=\exp\left(\int_{0}^{1}\log(\psi_{\Lambda}(f(t))\, dt\right).\]
Here $f$ is an arbitrary non-negative bounded Borel function on $[0,1]$.
In particular, when $f(s)=u1\!\!1_{[0,t]}(s)$, $u>0$, $t\in(0,1]$
the Laplace transform of $X(t)$ is given by\[
\mathbb{E}_{\mu_{\Lambda}}(e^{-uX(t)})=\exp\left(t\log(\psi_{\Lambda}(u))\right), ~t\in [0,1].\]
We call the pair $(X,\mu_{\Lambda})$ a realization of a L\'{e}vy process
with L\'{e}vy measure $\Lambda$ which is a subordinator, cf. \cite [Remark 2.1]{TVY2001}. 

Now we would like to highlight the link between tempered stable and stable subordinators. First of all it follows from (\ref{equiv}) 
that the L{\'e}vy measures $\Lambda_{\alpha}^{TSS}$ and $\Lambda^{SS}_{\alpha}$ are equivalent. Then we obtain from \cite[Theorem~33.1]{S} 
that $X_{\alpha}^{SS}$ and $X_{\alpha}^{TSS}$ have equivalent laws with density given in \cite[Theorem~33.2]{S}, see (\ref{eqdensity}) below. 
We notice that the authors in \cite{TVY2001,VY} constructed a family of measures, equivalent 
to $\alpha$-stable laws with given densities which converges weakly to the gamma measure. 
This is the content of the following remark.
\begin{remark}\label{TVYprocess}
Let $\tilde{X}_{\alpha}$ be a process such that the law $\tilde{\mu}_{\alpha}:=\mathcal{L} (\tilde{X}_{\alpha}(1))$ is 
equivalent to $\mu^{SS}_{\alpha}$ with density  
\begin{equation}\label{eqdensity}
\frac{d\tilde{\mu}_{\alpha}}{d\mu^{SS}_{\alpha}}(\eta)=\frac{\exp(-\alpha^{-1/\alpha}X(1)(\eta))}
{\mathbb{E}_{\mu^{SS}_{\alpha}}\left(\exp(-\alpha^{-1/\alpha}X(1)(\eta))\right)}
= e^{\alpha^{-1}}e^{-\alpha^{-1/\alpha}X(1)(\eta)}.
\end{equation}
Then the law of the tempered stable subordinator $X_{\alpha}^{TSS}$ is nothing but the
law of the process $\alpha^{-1/\alpha}\tilde{X}_{\alpha}$.
\end{remark}
\subsection{L{\'e}vy processes with finite variation paths}
We consider a L{\'e}vy process with the following triplet $(0,0,\Lambda)$.
We are interested here in the subclass of $\mathfrak{M}$ satisfying
\begin{eqnarray}
\Lambda(\mathbb{R}) & = & \infty,\label{Aout11eq1}\\
\int_{\left|s\right| \leq1}\left| s \right|\,d\Lambda(s) & < & + \infty.\label{Aout11eq2}
\end{eqnarray}
Condition (\ref{Aout11eq2}) means that the corresponding L{\'e}vy process has finite variation paths. 
\subsection*{Examples}
\begin{enumerate}
\item \textbf{~Stable process (S)}.
Symmetric $\alpha$-stable processes $X^{S}_{\alpha}$, with $\alpha \in (0,1)$, are the class of L{\'e}vy processes whose characteristic exponents 
correspond to those of symmetric $\alpha$-stable distributions. The corresponding L{\'e}vy measure is given by
\[
d\Lambda^{S}_{\alpha}(s)=\left(\frac{1}{\left|s\right|^{1+\alpha}}1\!\!1_{\{s<0\}}+
\frac{1}{s^{1+\alpha}}1\!\!1_{\{s>0\}}\right)ds.\]
The characteristic exponent $\varPsi_{\Lambda^{S}_{\alpha}}$ has the form
\[
\varPsi_{\Lambda^{S}_{\alpha}}(u)=-|u|^{\alpha},~ u\in\mathbb{R}.
\]

\item \textbf{~Tempered stable process (TS)}. 
It is well known that $\alpha$-stable distributions, with $\alpha\in(0,1)$, have infinite $p$-th moments
for all $p\geq \alpha$. This is due to the fact
that its L{\'e}vy density decays polynomially. Tempering the tails
with the exponential rate is one choice to ensure finite moments.
The tempered stable distribution is then obtained by taking a symmetric
$\alpha$-stable distribution and multiplying its L{\'e}vy measure
by an exponential functions on each half of the real axis. In explicit 

\[
d\Lambda^{TS}_{\alpha}(s)=\left(\frac{e^{-\left|s\right|}}{\left|s\right|^{1+\alpha}}1\!\!1_{\{s<0\}}+\frac{e^{-s}}{s^{1+\alpha}}1\!\!1_{\{s>0\}}\right)ds.\]
The characteristic exponent $\varPsi_{\Lambda^{TS}_{\alpha}}$ is given by
\[
\varPsi_{\Lambda^{TS}_{\alpha}}(u)=\Gamma(-\alpha)[(1-iu)^{\alpha}+(1+iu)^{\alpha}-2],~ u\in\mathbb{R}.
\]
The associated L{\'e}vy process will be called tempered stable process and denoted by $X_{\alpha}^{TS}$.

\item \textbf{~Modified tempered stable process (MTS).} 

The MTS distribution is obtained by taking an $\alpha$-stable law
with $\alpha\in(0,1/2)$ and multiplying the L{\'e}vy measure by a modified
Bessel function of the second kind on each side of the real
axis. It is infinitely divisible and has finite moments of all orders. It behaves asymptotically
like the $2\alpha$-stable distribution near zero and like the TS distribution on the tail. 
Then the L{\'e}vy density is given by 
\[
d\Lambda^{MTS}_{\alpha}(s)=\frac{1}{\pi}\left(\frac{K_{\alpha+\frac{1}{2}}(\left|s\right|)}{\left|s\right|^{\alpha+\frac{1}{2}}}1\!\!1_{\{s<0\}}+\frac{K_{\alpha+\frac{1}{2}}(s)}{s^{\alpha+\frac{1}{2}}}1\!\!1_{\{s>0\}}\right)ds.
\]
 $K_{\alpha+\frac{1}{2}}$ is the modified Bessel function of the second kind given by the following integral representation
\begin{equation}
K_{\alpha+\frac{1}{2}}(s)=
\frac{1}{2}\left(\frac{s}{2}\right)^{\alpha+\frac{1}{2}}\int_{0}^{+\infty}
\exp\left( -t- \frac{s^{2}}{4t}\right) t^{-\alpha-\frac{3}{2}}\, dt.\label{Aout26eq2} 
\end{equation}
The characteristic exponent has the form 
\[
\varPsi_{\Lambda^{MTS}_{\alpha}}(u)=\frac{1}{\sqrt{\pi}}\,2^{-\alpha-\frac{1}{2}}\,\Gamma(-\alpha)[(1+u^{2})^{\alpha}-1], ~u\in\mathbb{R}.
\]
The induced L{\'e}vy process, denoted by $X_{\alpha}^{MTS}$, will be called modified tempered stable process. 
For additional details on MTS distributions the reader may consult \cite{KRCB}. 
\end{enumerate}

\subsection{L{\'e}vy process of infinite variation paths}
Finally, we would like to consider a subclass of $\mathfrak{M}$ satisfying (\ref{Aout11eq1}) and the following condition
 
\begin{equation}
{\displaystyle \int_{\left|s\right| \leq1} \left| s \right|\,d\Lambda(s)}  =  \infty.\label{Aout26eq1}\end{equation} 

\subsection*{Examples}                                                                                                       
\begin{enumerate}
\item 
Symmetric $\alpha$-stable processes, tempered stable processes, with $\alpha \in (1,2)$ and 
modified tempered stable processes, with $\alpha \in (1/2,1)$.

\item\textbf{~Normal inverse Gaussian process (NIG).} 
The NIG distribution was introduced
in finance by Barndorff-Nielsen. It might be of interest to know that the NIG
distribution is a special case of the generalized hyperbolic
distribution, introduced also by Barndorff-Nielsen to model the logarithm
of particle size, see references below.

Let $\{X^{NIG}(t),\, t\in [0,1]\}$ be a L{\'e}vy process with L{\'e}vy measure given by
\[
 d\Lambda^{NIG}(s)=\frac{K_{1}(|s|)}{\pi |s|}ds,
\]
where $K_{1}$ is the modified Bessel function of the second
kind with index $1$. The characteristic exponent is equal to 
\[
\varPsi_{\Lambda^{\mathrm{NIG}}}(u)= \left(1-\sqrt{1+u^{2}}\right), u\in \mathbb{R}.
\]
 
The process $\{X^{NIG}(t),\, t\in [0,1]\}$ is a L{\'e}vy process with the triplet $(0,0,\Lambda^{NIG})$. 

For further results related to the normal inverse Gaussian distributions see
\cite{Br1,Br2}  and \cite{Ryd1,Ryd2}.
\end{enumerate}
We conclude this section with the following remark. 

\begin{remark}\label{7octrem}
All L{\'e}vy processes considered before are such that:
\begin{enumerate}
\item  Their paths belong to the set of all c{\`a}dl{\`a}g 
functions, denoted by $\mathbb{D}([0,1],\mathbb{R})$, i.e.~all real-valued right continuous with left limits functions on $[0,1]$.  
\item They are pure jump semimartingales processes without fixed times of discontinuity.
\end{enumerate}
\end{remark}

\section{Weak convergence and uniform tightness}
In this section at first we present a result on weak convergence of the above families in $\mathbb{D}([0,1],\mathbb{R})$ 
endowed with the Skorohod topology $\mathcal{J}_{1}$, ($\mathbb{D},\,\mathcal{J}_{1}$).
 This convergence will be denoted by  \textquotedblleft$\overset{\mathbb{D}}{\longrightarrow}$\textquotedblright. 
On the other hand, since we will deal with continuous limit processes, we are interested in the tightness and 
weak convergence in the space $\mathbb{D}([0,1],\mathbb{R})$ equipped with the uniform topology $\mathcal{U}$, $(\mathbb{D},\mathcal{U})$. 
We will denoted them by \textquotedblleft$\mathbb{C}$-tight\textquotedblright \,and 
\textquotedblleft$\overset{\mathbb{C}}{\longrightarrow}$\textquotedblright, respectively.
Finally, after recalling the definition of the uniform tightness as well as a useful criterion, 
we prove that the processes considered satisfy this condition, cf. Propositions \ref{prop1sept20} and \ref{prop2sept20} below.

\subsection{Weak convergence in $(\mathbb{D},\mathcal{J}_{1})$}

In this subsection we present the weak convergence in $(\mathbb{D},\mathcal{J}_{1})$ 
of the families of processes $\{X_\alpha^{TSS},\alpha\in(0,1/2)\}$ and $\{X_\alpha^{MSS},\;\alpha\in(0,1/2)\}$. 
We start with the following elementary lemma. 

\begin{lemma}
 We have the following weak convergence of the one dimensional law:

\begin{description}
 \item[(i)]  $X_{\alpha}^{TSS}(1)\, \overset{\mathcal{L}}{\longrightarrow} \, X_{\gamma}(1)$, $\alpha \rightarrow 0 $.
\item[(ii)]  $X_{\alpha}^{MTS}(1) \, \overset{\mathcal{L}}{\longrightarrow}\, X^{NIG}(1)$, $\alpha \rightarrow 1/2 $.
 \end{description}
\label{Re-TVY} \end{lemma}
\begin{proof}
 The result in (i) is a consequence of Proposition 6.3 in \cite{TVY2001}. 
\newline (ii)
 It is easy to see that the characteristic exponent $\varPsi_{\Lambda^{MTS}_{\alpha}}(1)$ converge
 to $\varPsi_{\Lambda^{NIG}}(1)$ when $\alpha$ goes to $1/2$. This implies that $X_{\alpha}^{MTS}(1)$ converge weakly to $X^{NIG}(1)$.
\end{proof}

\begin{proposition}\label{TVYconverge}
We have the following weak convergence in $(\mathbb{D},\mathcal{J}_{1})$:
\begin{description}
\item [(i)] $ X_{\alpha}^{TSS}\overset{\mathbb{D}}{\longrightarrow} 
X_{\gamma}$, as $\alpha \rightarrow 0$.
\item [(ii)] $X_{\alpha}^{MTS} \overset{\mathbb{D}}{\longrightarrow}
X^{NIG}$, as $\alpha \rightarrow 1/2$.
\end{description}
  \label{Re1-ES} \end{proposition}
\begin{proof}
Since L{\'e}vy processes are semimartingales with stationary independent increments, then 
it follows from \cite[Corollary~3.6]{JS} that the convergence of the marginal laws of $X_{\alpha}^{TSS}(1)$ and 
$X_{\alpha}^{MTS}(1)$ is equivalent to the weak convergence of processes $X_{\alpha}^{TSS}$ and 
$X_{\alpha}^{MTS}$ in $(\mathbb{D},\mathcal{J}_{1})$.
 \end{proof}

Now we are interested in the weak convergence 
of certain renormalization of pure jump subordinator.
Let $X$ be a subordinator with L{\'e}vy measure $\Lambda$ satisfying the conditions (\ref{Mai31eq3})-(\ref{Mai31eq4}) and 
$X_{\varepsilon}$ be the sum of its jumps of size in $(0,\varepsilon)$. Then the corresponding L{\'e}vy measure $\Lambda_{\varepsilon}$
is nothing but the restriction of $\Lambda$ to $(0,\varepsilon]$.  
We denote the expectation of $X_{\varepsilon}(1)$ by $\mu(\varepsilon):=\int_{(0,\varepsilon]}sd\Lambda(s)$.
We consider the renormalized process $Y_{\varepsilon}:=\mu(\varepsilon)^{-1}X_{\varepsilon}$ and state 
the following convergence result proved in \cite{covo}.
\begin{proposition}
 The following statements hold, as $\varepsilon\rightarrow0$.
\begin{description}
 \item[(i)] If $\mu(\varepsilon)/\varepsilon \rightarrow c$, where $0<c<+\infty$, then 
$Y_{\varepsilon}\overset{\mathbb{D}}{\longrightarrow}c^{-1}X^{*}_{c}$ where $X^{*}_{c}$ is a pure jump subordinator
with L{\'e}vy measure given by $d\Lambda^{*}_{c}(s)=1\!\!1_{(0,1]}(s)(c/s)ds$. 
\item[(ii)] If $\mu(\varepsilon)/\varepsilon \rightarrow +\infty$, then 
$Y_{\varepsilon}\overset{\mathbb{D}}{\longrightarrow}{\bf{t}}:=\left\{ t, t\in [0,1]\right\}$.
 \end{description}
\end{proposition}
\begin{remark}
 Since $Y_{\varepsilon}$ are L{\'e}vy processes and the limit process in the statement (ii) is continuous, then
 it follows from \cite[Theorem 19]{pollard} that the convergence holds also in $(\mathbb{D},\mathcal{U})$ as follows
\begin{description}
 \item[(ii)']If $\mu(\varepsilon)/\varepsilon \rightarrow +\infty$, then 
$Y_{\varepsilon}\overset{\mathbb{C}}{\longrightarrow}{\bf{t}}$.
\end{description}
\end{remark}

We give some examples of L{\'e}vy processes which illustrate the above proposition. 
 \begin{enumerate}
  \item Gamma process, $\mu(\varepsilon)/\varepsilon \rightarrow 1$.

\item Stable and tempered stable subordinators, $\alpha \in (0,1)$, $\mu(\varepsilon)/\varepsilon \rightarrow +\infty$.

\end{enumerate}

\subsection{Weak convergence in $(\mathbb{D},\mathcal{U})$}

In this subsection we are interested in the weak convergence 
of certain renormalizations of L{\'e}vy processes.
Let $X$ be a L{\'e}vy process with characteristic function of the form 
\[
\mathbb{E}(e^{iuX(t)})= \exp \left(t\left(i b u-\frac{1}{2}cu^{2}+
\int_{-\infty}^{+\infty}(e^{ius}-1-ius\,1\!\!1_{\{ \left|s\right|<1\} }(s))\, 
d\Lambda(s)\right)\right)
\]
where $t\in[0,1]$, $u\in \mathbb{R}$ and the L{\'e}vy measure $\Lambda$ does not have atoms in some neighbourhood of the origin.  
For each $\varepsilon\in (0,1)$, let us consider $\tilde{X}_{\varepsilon}$ the compensated sum 
 of jumps of $X$ taking values in $(-\varepsilon,\varepsilon)$. It is well known that $\{\tilde{X}_{\varepsilon},\, 0<\varepsilon\leq 1\}$ is a family of L{\'e}vy processes with characteristic function
\[
\mathbb{E}(e^{iu\tilde{X}_{\varepsilon}(t)})=\exp\left(t\int_{|s|\leq \varepsilon}(e^{ius}-1-ius)\, d\Lambda(s)\right),\quad t\in[0,1].\]
It is clear that, for each $\varepsilon>0$,  $\tilde{X}_{\varepsilon}$ is a martingale with jumps
 bounded by $\varepsilon$ with $\mathbb{E}(\tilde{X}_{\varepsilon}(1))=0$
and 
\[
 \mathbb{E}(\tilde{X}^{2}_{\varepsilon}(1))=\int_{|s|\leq \varepsilon}s^{2}d\Lambda(s)=:\sigma^{2}(\varepsilon).
\]
We consider the renormalization process $\tilde{Y}_{\varepsilon}:=\sigma(\varepsilon)^{-1}\tilde{X}_{\varepsilon}$ and state 
the following convergence result due to  \cite{As-Ro}.
\begin{proposition}\label{21juiprop}
The following  are equivalent
 \begin{enumerate}
  \item $\tilde{Y}_{\varepsilon}~{ \overset{\mathbb{C}}{\longrightarrow}}~W$ as $\varepsilon \rightarrow 0$, 
where $W$ is a standard Brownian motion.
\item $\dfrac{\sigma(\varepsilon)}{\varepsilon}\longrightarrow \infty$ as $\varepsilon \rightarrow 0$.
 \end{enumerate}

\end{proposition}
\begin{remark}
 For each $\varepsilon \in (0,1)$, $\tilde{Y}_{\varepsilon}$ is a L{\'e}vy process with characteristic function given by
\[ 
\mathbb{E}(e^{iu\tilde{Y}_{\varepsilon}(t)})=\exp\left(t\left[iub_{\varepsilon}+
\int_{\mathbb{R}}\left(e^{ius}-1-ius1\!\!1_{\{|s|\leq1\}}(s)\right)\, d\tilde{\Lambda}_{\varepsilon}(s)\right]\right),\, t\in[0,1],
\]
where the L{\'e}vy measure $\tilde{\Lambda}_{\varepsilon}$ is defined, for any $B\in \mathcal{B}(\mathbb{R})$, by
\begin{equation}\label{levreno}
 \tilde{\Lambda}_{\varepsilon}(B):=\Lambda(\sigma(\varepsilon) B\cap (-\varepsilon,\varepsilon)),  
\end{equation}
and 
\begin{equation}
 b_{\varepsilon}:=-\sigma(\varepsilon)^{-1}\int_{\sigma(\varepsilon)\wedge \varepsilon\leq|s|\leq \varepsilon}s\,d\Lambda(s). 
\end{equation}
\end{remark}

We give some examples of L{\'e}vy processes for which the above renormalization converge.  
 \begin{enumerate}
  \item Symmetric $\alpha$-stable processes, $\alpha \in (0,2)$, $\sigma(\varepsilon)=(2/(2-\alpha))^{1/2}\varepsilon^{1-\alpha/2}$. 
\item Tempered stable processes, $\alpha \in (0,1)$, $\sigma(\varepsilon)\geq (2/(2-\alpha))^{1/2}\varepsilon^{1-\alpha/2}e^{-\varepsilon/2}.$  
\item Modified tempered stable processes, $\alpha \in (0,1/2)$, $\sigma(\varepsilon)\approx (2/((2-2\alpha)\pi))^{1/2}\varepsilon^{1-\alpha}.$ 
\item Normal inverse Gaussian, $\sigma(\varepsilon)\approx (2/\pi)^{1/2}\varepsilon^{1/2}.$ 
 \end{enumerate}
We notice that the examples 1.~and 4.~above were considered in \cite{As-Ro}.
\subsection{\large{Uniform tightness of L{\'e}vy processes}}

First we recall the definition and criterion of the uniform tightness (\textbf{UT}) needed later on. 
The following definition was proposed by  \cite{JMP}.  
\begin{definition}\label{defUT}
A sequence of semimartingales $\{Z^{n},\,n\geq 1\}$ is said to be uniformly tight  
if for each $t\in (0,1]$, the set
 \[
  \left\{ \int^{t}_{0} H^{n}(s_{-})\,dZ^{n}(s), H^{n}\in \mathcal{H}, n \geq 1 \right\}
 \]
 is stochastically bounded (uniformly in $n$).
\end{definition}
In the above definition $\mathcal{H}$ denotes the collection of simple predictable processes of the form
\[
H(t)=H_{0}+\sum_{i=1}^{m}H_{i}1\!\! 1_{(t_{i},t_{i+1}]}(t),
\]
where $H_{i}$ is $\mathcal{F}_{t_{i}}$-measurable such that $|H_{i}|\leq 1$ and $0=t_{0}\leq \ldots\leq t_{m+1}=t$
 is a finite partition of $[0,t]$. 

In practice it is not easy to verify the (\textbf{UT}) condition as stated in Definition~\ref{defUT}.   
Thus we look for a more convenient criterion due to  \cite{KP3}. 
Let $Z$ be an adapted process with c{\`a}dl{\`a}g paths and $\{Z^{n}, n\in \mathbb{N}\}$ be a sequence of semimartingales, 
with the canonical decompositions
\begin{equation}
Z^{n}(t)=M^{n}(t)+A^{n}(t),\label{eq1Juin17}
\end{equation}
where $A^{n}$ is a predictable process with locally
bounded variation and $M^{n}$ is a (locally bounded) local martingale.
\begin{proposition}\label{KPUT}[cf. \cite{KP3}]
 Assume that $Z^{n}\overset{\mathbb{D}}{\longrightarrow}Z$ and one of the following two conditions holds 

\begin{equation}\label{UCV1}
 \sup_{n\in \mathbb{N}}\left\{\mathbb{E}\left([M^{n},M^{n}](1)+\int_{0}^{1}|dA^{n}(t)|\right)\right\}<+\infty, 
\end{equation}
\vspace{0.1cm}
\begin{equation}\label{UCV2}
 \sup_{n\in \mathbb{N}}\left\{\mathbb{E}\left(\sup_{t\leq 1}|\Delta M^{n}(t)|+\int_{0}^{1}|dA^{n}(t)|\right)\right\}<+\infty. 
\end{equation}

Then $\{Z^{n}, n\in \mathbb{N}\}$ satisfies (\textbf{UT}). 
\end{proposition}

\begin{remark}
\begin{enumerate}
\item If $Z$ is a continuous semimartingale then we assume that $Z^{n}\overset{\mathbb{C}}{\longrightarrow}Z$. 
 \item The conditions (\ref{UCV1}) and (\ref{UCV2}) imply the uniform controlled variation (\textbf{UCV}) of $\{Z^{n}, n\in \mathbb{N}\}$
 introduced in \cite{KP3}. 
\item Since $Z^{n}\overset{\mathbb{D}}{\longrightarrow}Z$ then the (\textbf{UT}) and (\textbf{UCV}) are equivalent, see \cite{KP3}.
\end{enumerate}

\end{remark}

Next, we are interested in the decomposition (\ref{eq1Juin17}) for a L{\'e}vy process $Z$. 
We start by splitting $Z$ into two parts depending on the size of the jumps:
\[
Z(t)=R(t)+N(t)
\]
with $N(t)=\sum_{s\leq t}\Delta Z(s)~1\!\!1_{\left\lbrace\left|\Delta Z(s)\right|>1\right\rbrace}$
and $R$ with jumps bounded by $1$. Since $R$ is a L{\'e}vy process with bounded jumps its canonical decomposition
is, by means of \cite[pp.~103]{A}, of the simple form $R(t)=R_{0}(t)+t\mathbb{E}(R(1))$
where $\{ R_{0}(t):t\in\left[0,1\right]\}$ is a c{\`a}dl{\`a}g centred
square-integrable martingale with jumps bounded by $1$. Hence the decomposition (\ref{eq1Juin17}) takes the form 
\begin{equation}
 Z(t)=R_{0}(t)+t\mathbb{E}(R(1))+\sum_{s\leq t}\Delta Z(s)~1\!\!1_{\{\left|\Delta Z(s)\right|>1\}}.
\label{Aout13eq1}
\end{equation}

Now we are able to state the main result of this subsection.
\begin{proposition}\label{prop1sept20}

 The following families  satisfy (\textbf{UT})
\begin{description}
\item[(i)] $\{ X_{\alpha}^{TSS},\, \alpha \in (0,1/2) \},$ 
\item[(ii)] $ \{ X_{\alpha}^{MTS},\, \alpha \in (0,1/2) \}.$
\end{description}
\end{proposition}
\begin{proof}

Since the families $\{ X_{\alpha}^{TSS},\, \alpha \in (0,1/2) \}$ and 
$\{ X_{\alpha}^{MTS},\, \alpha \in (0,1/2) \}$ are weakly convergent, then 
in order to obtain the (\textbf{UT}) property, 
we have only to check condition (\ref{UCV1}) of Proposition \ref{KPUT}.

(i) The decomposition (\ref{Aout13eq1}) for the process $X_{\alpha}^{TSS}$ 
is given by
\begin{equation}
X_{\alpha}^{TSS}(t)=R_{\alpha,0}^{TSS}(t)+t\mathbb{E}(R_{\alpha}^{TSS}(1))
+\sum_{s\leq t}\Delta X_{\alpha}^{TSS}1\!\!1_{\{\left|\Delta X_{\alpha}^{TSS}\right|>1\}}.
\end{equation}
Thus, condition (\ref{UCV1}) becomes
\[\sup_{\alpha \in (0,1/2)}\left(\int_{0}^{1}s^{2}~d\Lambda^{TSS}_{\alpha}(s)+\int_{0}^{+\infty}s~d\Lambda^{TSS}_{\alpha}(s)\right)
<+\infty,
\]
which is simple to verify.

(ii)  It is easy to see that $\mathbb{E}(R_{\alpha}^{MTS}(1))=0$. 
Then the (\textbf{UT}) condition follows from 
\[\sup_{\alpha \in (0,1/2)}\left(\int_{|s|\leq 1}s^{2}~d\Lambda^{MTS}_{\alpha}(s)+\int_{|s|>1}|s|~d\Lambda^{MTS}_{\alpha}(s)\right)
<+\infty.
\]
To show this we use the integral representation (\ref{Aout26eq1}) for the Bessel function $K_{\alpha+1/2}$ and estimate the above integrals as
\begin{eqnarray*}
\int_{|s|> 1}|s|~d\Lambda^{MTS}_{\alpha}(s)&= & 2^{-\alpha-1/2}\int_{1}^{+\infty}\int_{0}^{+\infty}s~e^{-\frac{s^{2}}{4t}}
e^{-t}t^{-(\alpha+3/2)}dtds \\
&=&2^{1/2-\alpha}\int_{0}^{+\infty}e^{-(t+\frac{1}{4t})}t^{-\alpha-1/2}dt\\
&\leq& 5~ 2^{1/2-\alpha}\int_{1/4}^{+\infty}e^{-(t+\frac{1}{4t})}dt.
\end{eqnarray*}

\begin{eqnarray*}
\int_{|s|\leq 1}s^{2}~d\Lambda^{MTS}_{\alpha}(s)&\leq & 2\int_{0}^{+\infty}s^{2}~d\Lambda^{MTS}_{\alpha}(s) \\
&=&\sqrt{\pi}~ 2^{1/2-\alpha}~ \Gamma(1-\alpha).
\end{eqnarray*}
 This completes the proof.
\end{proof}
Next we state the (\textbf{UT}) property for the renormalized families
$\{ Y_{\varepsilon},\, \varepsilon \in (0,1) \}$ and $\{ \tilde{Y}_{\varepsilon},\, \varepsilon \in (0,1) \}$.
\begin{proposition}\label{prop2sept20}
\begin{description}
 \item [(i)] Assume that $\mu(\varepsilon)/\varepsilon$ converges in $(0,+\infty]$. Then the 
renormalized family $\{ Y_{\varepsilon},\, \varepsilon \in (0,1) \}$ satisfies (\textbf{UT}).
\item[(ii)] Assume that $\tilde{Y}_{\varepsilon}{ \overset{\mathbb{C}}{\longrightarrow}}~W$. 
Then the renormalized family $\{ \tilde{Y}_{\varepsilon},\, \varepsilon \in (0,1) \}$ satisfies (\textbf{UT}).
\end{description} 
\end{proposition}
\begin{proof}
(i) Since the process $Y_{\varepsilon}$ is a pure jump subordinator, then the condition (\ref{UCV1}) becomes

\begin{equation}
\sup_{\varepsilon\in (0,1)}\left\{\mathbb{E}\left(\int_{0}^{1}|dY_{\varepsilon}(t)|\right)\right\}= 
\sup_{\varepsilon\in (0,1)}\mathbb{E}\left(Y_{\varepsilon}(1)\right)=1.
\end{equation}
So the (\textbf{UT}) condition is a consequence of Proposition \ref{KPUT}.
 
\noindent (ii) First notice that, for each $\varepsilon \in (0,1)$, $\tilde{Y}_{\varepsilon}$ is a martingale with jumps 
bounded by $\varepsilon/\sigma(\varepsilon)$. Thus we obtain
\[
\mathbb{E}\left( \sup_{t\leq 1}|\Delta \tilde{Y}_{\varepsilon}(t)| \right)\leq \frac{\varepsilon}{\sigma(\varepsilon)}. 
\]
As a consequence of statement 2 of Proposition \ref{21juiprop} we have 
\[
\sup_{\varepsilon \in (0,1)}\mathbb{E}\left( \sup_{t\leq 1}|\Delta \tilde{Y}_{\varepsilon}(t)| \right)<\infty,
\]
which implies that condition (\ref{UCV2}) is satisfied.
Since $\tilde{Y}_{\varepsilon}$ is weakly convergent, then (\textbf{UT}) condition follows from Proposition \ref{KPUT}. 
\end{proof}

\section{ $\alpha$-Continuity of SDEs dri\-ven by L\'{e}vy processes}

The previous section established the weak convergence and uniform tightness for certain families of L{\'e}vy processes.  
Now we would like to apply these results to study the continuous dependence 
 problem for SDEs driven by these families of L{\'e}vy processes. 
For a survey on SDEs driven by L\'{e}vy processes we refer to \cite{B-S}. 
To begin, we give some notations useful in the sequel: for each $n\in \{2,3, \ldots\}$, 
\textquotedblleft$\overset{\mathbb{D}^n}{\longrightarrow}$\textquotedblright\, 
and \textquotedblleft$\mathbb{D}^{n}$-tight\textquotedblright\, denote
the weak convergence and tightness in $\mathbb{D}([0,1],\mathbb{R}^{n})$ 
endowed with the Skorohod topology. 
In the same way
\textquotedblleft$\overset{\mathbb{C}^n}{\longrightarrow}$\textquotedblright\, 
and \textquotedblleft$\mathbb{C}^{n}$-tight\textquotedblright\, denote 
the weak convergence and tightness for the uniform topology.  
\subsection{The modified tempered stable case}

We will make the following assumptions

\begin{description}
\item[(H.1)]$a_{\alpha},h_{\alpha}:\mathbb{R}\longrightarrow \mathbb{R}$ are continuous such that
$|a_{\alpha}(x)|+|h_{\alpha}(x)|\leq K(1+|x|)$ for all $\alpha \in (0,1/2), ~x \in \mathbb{R}$.
\item [(H.2)] The family $a_{\alpha}$ (resp.~$h_{\alpha}$) converge uniformly to 
$a$ (resp.~$h$) on each compact set in $\mathbb{R}$, as $\alpha\rightarrow0$. 
 \end{description}

We consider the following SDEs 
\begin{equation}
dY_{\alpha}^{TSS}(t)=a_{\alpha}(Y_{\alpha}^{TSS}(t_{-}))dX_{\alpha}^{TSS}(t)+
h_{\alpha}(Y_{\alpha}^{TSS}(t))dt,\quad Y_{\alpha}^{TSS}(0)=0,\label{fev09-eq1}
\end{equation} 
and 
 \begin{equation}
dY(t)=a(Y(t_{-}))dX_{\gamma}(t)+h(Y(t))dt,\quad Y(0)=0.\label{fev09-eq4}
\end{equation} 
\begin{remark}\label{exlem}
\begin{enumerate}
 \item Under the assumption \textbf{(H.1)}, for each $\alpha\in(0,1/2)$, the equation (\ref{fev09-eq1}) 
admits a weak solution, see  \cite{JM}.
\item Since the coefficients $a_{\alpha}$ and $a$ are not Lipschitz, then we do not have uniqueness of solutions
 for either equation (\ref{fev09-eq1}) or equation (\ref{fev09-eq4}).
\end{enumerate}
\end{remark}

The first $\alpha$-continuity result concerns the class of tempered stable subordinators. 

\begin{theorem}
\label{fev09-eq3} Under the assumptions \textbf{(H.1)-(H.2)} we have 
\begin{enumerate}
 \item The family of processes $(Y_{\alpha}^{TSS},X_{\alpha}^{TSS})$ is $\mathbb{D}^{2}$-tight.
\item Any limit point $(Y,X_{\gamma})$ of the family $(Y_{\alpha}^{TSS},X_{\alpha}^{TSS})$
satisfies equation (\ref{fev09-eq4}).
\item If uniqueness in law holds for the equation (\ref{fev09-eq4}), then 
\[
(Y_{\alpha}^{TSS},X_{\alpha}^{TSS}){ \overset{\mathbb{D}^{2}}{\longrightarrow}} (Y,X_{\gamma}),\quad \alpha\rightarrow0.
\] 
\end{enumerate}
\end{theorem}

\begin{proof}
1. At first we show that the family $\{ Y_{\alpha}^{TSS},\, \alpha\in(0,1/2)\}$ verify the (\textbf{UT}) condition.
For that we need to prove that the family $\{ \sup_{s\in [0,1]}|Y_{\alpha}^{TSS}(s)|,\, \alpha\in(0,1/2)\}$ 
is bounded in probability. Indeed, as $Y_{\alpha}^{TSS}$ satisfies the equation (\ref{fev09-eq1}) then, for $t\in [0,1]$, we have  
\begin{equation}
Y_{\alpha}^{TSS}(t)=\int_{0}^{t}a_{\alpha}(Y_{\alpha}^{TSS}(s_{-}))dX_{\alpha}^{TSS}(s)+
\int_{0}^{t}h_{\alpha}(Y_{\alpha}^{TSS}(s))ds.\label{eq1jan28}
\end{equation}
The set $\left\{ t:\, Y_{\alpha}^{TSS}(t)\neq Y_{\alpha}^{TSS}(t_{-})\right\} $ is countable
and so far it is Lebesgue negligible. Owing to this fact we may replace
$h_{\alpha}(Y_{\alpha}^{TSS}(t))$ by $h_{\alpha}(Y_{\alpha}^{TSS}(t_{-}))$ in the right-hand side of (\ref{eq1jan28}) and obtain
\begin{equation*}
Y_{\alpha}^{TSS}(t)=\int_{0}^{t}a_{\alpha}(Y_{\alpha}^{TSS}(s_{-}))dX_{\alpha}^{TSS}(s)+
\int_{0}^{t}h_{\alpha}(Y_{\alpha}^{TSS}(s_{-}))ds.
\end{equation*}
Using assumption \textbf{(H.1)} we get

\begin{equation*}
|Y_{\alpha}^{TSS}(t)|  \leq  K\int_{0}^{t}(1+|Y_{\alpha}^{TSS}(s_{-})|)d\{s+X_{\alpha}^{TSS}(s)\}, ~\forall\, t\in [0,1]. 
\end{equation*}
It follows from a Gronwall type inequality, see \cite[pp. 352]{Pr}, that
\begin{equation*}
|Y_{\alpha}^{TSS}(t)|  \leq  K\exp(t+X_{\alpha}^{TSS}(t)), ~\forall\, t\in [0,1]. 
\end{equation*}
Now since $X_{\alpha}^{TSS}$ is increasing it yields
\begin{equation*}
\sup_{t\in [0,1]}|Y_{\alpha}^{TSS}(t)|  \leq  K\exp(1+X_{\alpha}^{TSS}(1)). 
\end{equation*}
We infer the boundedness in probability of the family $\{ \sup_{s\in [0,1]}|Y_{\alpha}^{TSS}(s)|,\, \alpha\in(0,1/2)\}$ 
 from the boundedness in probability of the the family $\{ X_{\alpha}^{TSS}(1),\, \alpha\in(0,1/2)\}$
which is a consequence of the uniform tightness of the family $\{ X_{\alpha}^{TSS},\, \alpha\in(0,1/2)\}$, 
see Lemma 1.2 of \cite{JMP} or \cite{S84}. 

Hence the family $\{ \sup_{s\in [0,1]}|a_{\alpha}(Y_{\alpha}^{TSS}(s))|,\, \alpha\in(0,1/2)\}$ 
(resp. $\{ \sup_{s\in [0,1]}|h_{\alpha}(Y_{\alpha}^{TSS}(s))|,\, \alpha\in(0,1/2)\}$) 
 is also bounded in probability since $a_{\alpha}$ (resp.~$h_{\alpha}$) has at most linear growth. 
Therefore it is easy to see that the family 
$\{\int_{0}^{\cdot}h_{\alpha}(Y_{\alpha}^{TSS}(t))dt,\, \alpha\in(0,1/2)\}$ 
satisfies the (\textbf{UT}) condition.
On the other hand, the uniform tightness of the family 
$\{\int_{0}^{\cdot}a_{\alpha}(Y_{\alpha}^{TSS}(t))dX_{\alpha}^{TSS}(t),\allowbreak \alpha\in(0,1/2)\}$ follows from \cite[Lemme 1-6]{MS}.
As a consequence we get the (\textbf{UT}) condition for the family $\{ Y_{\alpha}^{TSS},\, \alpha\in(0,1/2)\}$.

On the next step we show that the family of processes $(Y_{\alpha}^{TSS},X_{\alpha}^{TSS})$ is $\mathbb{D}^{2}$-tight.
Since the function $a_{\alpha}$ is continuous we can always find a sequence of $C^{2}$ functions,
$\{a_{\alpha,n},\, n\in \mathbb{N}\}$, which approximate uniformly $a_{\alpha}$
 on compact sets of $\mathbb{R}$. Now let us consider the sequence of process $Y_{\alpha,n}^{TSS}$ defined by
\begin{equation}
dY_{\alpha,n}^{TSS}(t)=a_{\alpha,n}(Y_{\alpha}^{TSS}(t_{-}))dX_{\alpha}^{TSS}(t)+
h_{\alpha}(Y_{\alpha}^{TSS}(t))dt,\quad Y_{\alpha}^{TSS}(0)=0.
\end{equation} 
As the function $a_{\alpha,n}$ is of class $C^{2}$ then we get from \cite[Lemme 1-7]{MS} that 
the family $\{ a_{\alpha,n}(Y_{\alpha}^{TSS}), \alpha\in(0,1/2)\}$ is uniformly tight.
Now it follows from \cite[Proposition 3-3]{MS} (see also \cite{KP1991}) that the family of processes
$(\int_{0}^{\cdot}a_{\alpha,n}(Y_{\alpha}^{TSS}(t))dX_{\alpha}^{TSS}(t), \allowbreak X_{\alpha}^{TSS})$  
is $\mathbb{D}^{2}$-tight and consequently $(Y_{\alpha,n}^{TSS},X_{\alpha}^{TSS})$ is also $\mathbb{D}^{2}$-tight. 
It is simple to see that 
\[
 \lim_{n\rightarrow \infty} P\left[ \sup_{t\leq1}|Y_{\alpha,n}^{TSS}(t)-Y_{\alpha}^{TSS}(t)|>\delta\right]=0
\]
for all $\delta>0$. 
Then we use again \cite[Proposition 3-3]{MS} to obtain that the family of processes
 $(Y_{\alpha}^{TSS},X_{\alpha}^{TSS})$ is $\mathbb{D}^{2}$-tight.

The proof of both assertions 2 and 3 is similar to the one of \cite[Th{\'e}or{\`e}me 3.5]{MS}, therefore we omit it.
\end{proof}
We now turn to an example of equation (\ref{fev09-eq4}), for which there is no uniqueness in law.
\begin{example} 
We consider the following equation
\begin{equation}
dY(t)=dX_{\gamma}(t)+h(Y(t))dt,\quad Y(0)=0,\label{eq1-fev8}
\end{equation} 
where $h$ is bounded continuous and equal to $ \mathrm{sign}(x)|x|^{\beta}$ in some
neighborhood of $x=0$ for certain positive constant $\beta<1$.

This example is inspired by the work of of \cite{Tan1},
who treats the problem of uniqueness of the equation (\ref{eq1-fev8}) with the symmetric stable process instead of the gamma process.
Precisely, the authors show that the equation (\ref{eq1-fev8}), with the drift $h$ given above, does not admit the uniqueness property.
The key of the proof is the asymptotic rate of growth of the sample paths of the symmetric stable process at the origin.
In our case the gamma process satisfies the following short time behavior 
\begin{equation}
 \lim_{t\downarrow0}~\frac{X_{\gamma}(t)}{t^{1/(1-\beta)}}=0 \quad a.s.
\end{equation}
which is a consequence of the finiteness of $\displaystyle\int_{0}^{1}\Lambda_{\gamma}[t^{1/(1-\beta)},+\infty)~dt$,
 see \cite{Ber1} or \cite{S}. 
 This is sufficient to show non-uniqueness proceeding along the lines as in  \cite{Tan1}, Theorem 3.2.
\end{example}

In a similar way we obtain an analogous $\alpha$-continuity result if we replace the processes $X_{\alpha}^{TSS}$ and $X_{\gamma}$ 
in equations (\ref{fev09-eq1}) and (\ref{fev09-eq4})
by $X_{\alpha}^{MTS}$ and $X^{NIG}$ respectively and assumption \textbf{(H.2)} by
\begin{description}
\item [(H.$\bf{2}^\prime$)] The family $a_{\alpha}$ (resp. $h_{\alpha}$) converge uniformly to 
$a$ (resp. $h$) on each compact set in $\mathbb{R}$, as $\alpha\rightarrow 1/2$. 
\end{description}

We state this in the following theorem.

\begin{theorem}
Under the assumptions \textbf{(H.1)} and \textbf{(H.$\bf{2}^\prime$)} we have 
\begin{enumerate}
 \item The family of processes $(Z_{\alpha}^{MTS},X_{\alpha}^{MTS})$ with 
\begin{equation}
dZ_{\alpha}^{MTS}(t)=a_{\alpha}(Z_{\alpha}^{MTS}(t_{-}))\,dX_{\alpha}^{MTS}(t)+h_{\alpha}(Z_{\alpha}^{MTS}(t))\,dt
,\quad Z_{\alpha}^{MTS}(0)=0,\label{Sept09-eq1}
\end{equation}
is $\mathbb{D}^{2}$-tight.
\item  Any limit point $(Z,X^{NIG})$ of the family $(Z_{\alpha}^{MTS},X_{\alpha}^{MTS})$
satisfies equation 
\begin{equation}
dZ(t)=a(Z(t_{-}))\,dX^{NIG}(t)+h(Z(t))\,dt,\quad Z(0)=0. \label{eqnig} 
\end{equation}
\item If uniqueness in law holds for equation (\ref{eqnig}), then 
\[
(Z_{\alpha}^{MTS},X_{\alpha}^{MTS}){ \overset{\mathbb{D}^{2}}{\longrightarrow}} (Z,X^{NIG}),\quad
  \alpha\rightarrow1/2.
\]
\end{enumerate}
\end{theorem}

\subsection{The renormalized case}
Finally, we conclude the section presenting a $\varepsilon$-continuity result for SDEs driving by the renormalized families 
$\{ Y_{\varepsilon},\, \tilde{Y}_{\varepsilon},\, \varepsilon \in (0,1)\}$.
 To do so, let us consider the following equations
\begin{equation}
 dZ_{\varepsilon}(t)=a_{\varepsilon}(Z_{\varepsilon}(t_{-}))\,d\tilde{Y}_{\varepsilon}(t)+
h_{\varepsilon}(Z_{\varepsilon}(t))dY_{\varepsilon}(t)
,\quad Z_{\varepsilon}(0)=0,\label{21juieq1}
\end{equation}
and
\begin{equation}
 dZ(t)=a(Z(t))\,dW(t)+h(Z(t))\,dt,\quad Z(0)=0,\label{21juieq2}
\end{equation}
Our result then is stated in the following theorem.
\begin{theorem}
 Assume that 
\begin{description}
\item[(i)] $\mu(\varepsilon)/\varepsilon \rightarrow +\infty$, $\varepsilon \rightarrow 0$;
\item[(ii)] $\tilde{Y}_{\varepsilon}{ \overset{\mathbb{C}}{\longrightarrow}}~W$, $\varepsilon \rightarrow 0$; 
\item[(iii)] the families $\{ Y_{\varepsilon},\,\varepsilon \in (0,1)\}$ and
$\{ \tilde{Y}_{\varepsilon},\, \varepsilon \in (0,1)\}$ are independent;
\item[(iv)] the coefficients $h_{\varepsilon}$, 
$a_{\varepsilon}$ and $h$, $a$ satisfy the assumptions \textbf{(H.1)-(H.2)}.
\end{description} 
Then we have
\begin{enumerate}
 \item The family $\{(Z_{\varepsilon},\tilde{Y}_{\varepsilon}, Y_{\varepsilon}),~\varepsilon\in (0,1)\}$ is $\mathbb{C}^{3}$-tight.
\item  Any limit point $(Z,W,{\bf{t}})$ of the family $(Z_{\varepsilon},\tilde{Y}_{\varepsilon}, Y_{\varepsilon})$
satisfies equation (\ref{21juieq2}).

\item If uniqueness in law holds for equation (\ref{21juieq2}) then
 \[
(Z_{\varepsilon},\tilde{Y}_{\varepsilon}, Y_{\varepsilon}){ \overset{\mathbb{C}^{3}}{\longrightarrow}} (Z,W,{\bf{t}}),\quad
  \varepsilon \rightarrow 0.
\]
\end{enumerate}
\end{theorem}
\begin{proof}
First we know that $\{Y_{\varepsilon},\, \varepsilon\in (0,1)\}$ 
(resp.~$\{\tilde{Y}_{\varepsilon},\, \varepsilon\in (0,1)\}$) is a family of increasing processes (resp.~martingales)
 which converges to the continuous increasing process ${\bf{t}}$ (resp. to the continuous martingale $W$).
Since the two families are independents, then we have the following weak convergence 
\[
(Y_{\varepsilon}, \tilde{Y}_{\varepsilon}){ \overset{\mathbb{C}}{\longrightarrow}} ({\bf{t}},W),\quad
  \varepsilon \rightarrow 0. 
\]
Secondly, it is known that under (iv) equations (\ref{21juieq1}) and (\ref{21juieq2}) admit a weak solutions, see \cite[Theorem 1.8]{JM}.
Using the fact that $\sigma(\varepsilon)/ \varepsilon \longrightarrow \infty$ as $\varepsilon\rightarrow0$, we have 

\[
\int_{|s|>1}|s|d\tilde{\Lambda}_{\varepsilon}(s)=
(\sigma(\varepsilon))^{-1}\int_{\sigma(\varepsilon)<|s| \leq \varepsilon}|s|d\Lambda(s)\longrightarrow0.
\]

Finally the assumption \textbf{(H.1)} is sufficient for the continuity in the Skorohod space, 
cf. \cite[Example 5.3]{KP1991}. So the assertions 1-3 follow from \cite[Th{\'e}or{\`e}me 2.10]{MS}.
\end{proof}

\subsection*{Acknowledgement}

Financial support by PTDC/MAT/67965/2006 and FCT, POCTI-219 are gratefully acknowledged.
We are grateful to the anonymous referee for valuable suggestions on the first version of the paper.

\end{document}